\documentclass{amsart}
\usepackage{amsfonts,amssymb,latexsym,amsthm,amsmath}

\usepackage{amsfonts}
\usepackage{amscd,amsmath}
\usepackage{euscript}

\usepackage[arrow,matrix,curve]{xy}\SilentMatrices
\def\xyma{\xymatrix@M.7em}

\newtheorem{theorem}{Theorem}
\newtheorem{lemma}[theorem]{Lemma}

\newtheorem{proposition} [theorem]{Proposition}

\newtheorem{corollary}[theorem]{Corollary}

\DeclareMathAlphabet\mathbb{U}{msb}{m}{n}

\newcommand{\ilimit}{\,\varprojlim{}\:}
\newcommand{\Para}{\par\vspace{.5cm}\noindent}
\begin{document}
\title{Generalized dimension subgroups and derived functors}
\author{Roman Mikhailov and Inder Bir S. Passi*}\footnote{*Corresponding author}
\maketitle\centerline{\it Dedicated to Donald S. Passman on his seventy-fifth birthday}

\begin{quote}
\begin{abstract}Every two-sided ideal $\mathfrak a$ in the integral group ring $\mathbb Z[G] $ of a group $G$
determines a normal subgroup \linebreak $G \cap (1 + \mathfrak a)$ of $G$. In this paper certain problems
related to the identification of such subgroups, and their relationship with
derived functors in the sense of Dold-Puppe, are discussed.\end{abstract}
\end{quote}\par\vspace{.25cm}\noindent
\par\vspace{.25cm}\noindent MSC2010: 18A30, 18E25,
20C05, 20C07 \par\vspace{1cm}
\section*{Introduction} Let $G$ be a group and $\mathbb Z[G]$ its integral group ring.
Every two-sided ideal $\mathfrak a$ in the integral group ring $\mathbb Z[G] $ of a group $G$
determines a normal subgroup  $G \cap (1 + \mathfrak a)$ of $G$.  The identification of such normal subgroups is a fundamental problem in the theory of group rings.  On expressing the group $G$ as a quotient  of a free group $F$, the problem can clearly be viewed as an  identification problem in the free group ring  $\mathbb Z[F]$. Let  $\mathfrak g$ denote the augmentation ideal of  $\mathbb Z[G]$. In analogy with the case when $\mathfrak a$ is a power $\mathfrak g^n$, $n\geq 1$, of the augmentation ideal  and the corresponding subgroup, denoted $D_n(G)$,  is called the $n$th {\it dimension subgroup} of $G$, we call the normal subgroups $G\cap (1+\mathfrak a+\mathfrak g^n),\,n\geq 1,$ {\it generalized dimension subgroups}.   For every group $G$ and integer $n\geq 1$, it is easily seen that $D_n(G)\geq \gamma_n(G)$, the $n$th term in the lower central series of $G$. It is well-known that $D_3(G)=\gamma_3(G)$, or, equivalently, that $F\cap (1+\mathfrak r\mathbb Z[F]+\mathfrak f^3)=R\gamma_3(F)$ for every free presentation $G\cong F/R$. A relationship between generalized dimension subgroups and  derived functors in the sense of Dold-Puppe \cite{DoldPuppe} is noticed on considering the subgroup $F\cap (1+\mathfrak r\mathfrak f+\mathfrak f^3)$ (Theorem \ref{D(3)}; \cite{HMP}, Theorem 3.3). Motivated by this observation, our aim in this paper is to explore further    relations between generalized dimension subgroups and
derived functors.  We expect our approach will initiate further work leading to deeper understanding of generalized dimension subgroups.

\par\vspace{.25cm}Based on the investigations of Narain Gupta \cite{Gupta:87},  we identify, in Section 1,  \linebreak generalized dimension subgroups $F\cap (1+\mathfrak a +\mathfrak f^4)$  for various two-sided ideals $\mathfrak a$ in a free group ring $\mathbb Z[F]$. In view of the fact that the dimension series  $\{D_n(G)\}_{n\geq 1}$ and the lower central series $\{\gamma_n(G)\}_{n\geq 1}$, in general, cease to be identical from $n=4$ onwards, the identification of these subgroups becomes of particular interest.   \linebreak  Typical of the cases that we consider is the identification of the subgroup \linebreak $F\cap (1+\mathfrak f\mathfrak r\mathfrak f+\mathfrak f^4)$ (Theorem \ref{typical}).  In Corollary \ref{Guptaprob} we show
that $F\cap (1+ \mathfrak r^2\mathbb Z[F]+\mathfrak f^4)\leq R\gamma_4(F)$, thus  answering  Narain
Gupta's Problem 6.9(a) in \cite{Gupta:87}. This \linebreak result, however, leaves open for future investigation  the interesting question whether $F\cap (1+\mathfrak r^{n-1}\mathbb Z[F]+\mathfrak f^{n+1})\leq R\gamma_{n+1}(F)$ for some $n\geq 4$.
\par\vspace{.25cm}
 In Section 2 we exhibit how some of the generalized dimension subgroups   are \linebreak related to derived functors.
Our main result here  is Theorem \ref{main1} relating  \linebreak
$F\cap (1+\mathfrak f\mathfrak r\mathfrak f+\mathfrak f^4)$ to the
homology of the Eilenberg-MacLane space $K(G_{ab},\, 2)$, where
$G_{ab}$ is the abelianization of the group $G\cong F/R$. More
precisely, we prove that, if $G_{ab}$ is 2-torsion-free, then  $\frac{F\cap(1+\mathfrak f\mathfrak r\mathfrak f+\mathfrak f^4)}{[R,\,R,\,F]\gamma_4(F)}\cong L_1\operatorname{SP}^3(G_{ab})$, where $L_1\operatorname{SP}^3$ is the first derived functor of the third symmetric power functor, and
there is a natural short exact sequence
\begin{equation}\label{frfsq}
\frac{F\cap(1+\mathfrak f\mathfrak r\mathfrak f+\mathfrak f^4)}{[R,\,R,\,F]\gamma_4(F)}\hookrightarrow
H_7K(G_{ab},\,2)\twoheadrightarrow \operatorname{Tor}(G_{ab},\,\mathbb Z/3\mathbb Z).
\end{equation}
Consequently, if $G_{ab}$ is $p$-torsion-free for $p=2,\ 3$, then there is
a natural \linebreak isomorphism
\begin{equation}\label{frfis}
\frac{F\cap (1+\mathfrak f\mathfrak r\mathfrak f+\mathfrak
f^4)}{[R,\,R,\,F]\gamma_4(F)}\cong H_7K(G_{ab},\,2).
\end{equation}
We define a functor  $\mathfrak L_s^3(A)$ on the category $\mathfrak A$ of abelian groups (see Section 2 for \linebreak definition) and prove (Theorem \ref{super}) that if $G$ is a group and $F/R $ its free presentation, then
$$\frac{F\cap (1+\mathfrak r^2\mathfrak f+\mathfrak f^4)}{\gamma_3(R)\gamma_4(F)}\cong L_2\mathfrak L_s^3(G_{ab}).$$
\par\vspace{.25cm}
In Section 3 we study  limits of functors  from  the category  of free presentations of groups to the category of abelian groups, and show their  connection with quadratic and cubic functors. Our main results on limits are Theorems \ref{quadraticpicture} and \ref{cubicpicture} where we show, in particular,  that $\ilimit \frac{\gamma_2(F)}{\gamma_2(R)\gamma_3(F)}$, $\ilimit \frac{\gamma_3(F)}{\gamma_3(R)\gamma_4(F)}$ and $\ilimit \frac{\gamma_3(F)}{[\gamma_2(R),\,F]\gamma_4(F)}$  agree with certain derived functors.  In particular, we show (see Theorem
\ref{cubicpicture}) that $$ \ilimit
\frac{\gamma_3(F)}{[R,\,R,\,F]\gamma_4(F)}\cong L_1\operatorname{SP}^3(G_{ab}).
$$
\par\vspace{.25cm}
For background results, we refer the reader to the monographs \cite{Gupta:87} and \cite{MP:2009} and the article \cite{IvanovMikhailov}.

\section{Generalized dimension subgroups}

\par\vspace{.25cm}
Let $F$ be a free group with basis $X$. Then the augmentation ideal $\mathfrak f$ of the group ring $\mathbb Z[F]$ is a two-sided ideal which is free as a left (resp. right) $\mathbb Z[F]$-module with basis $\{x-1\,|\,x\in X\}$ (\cite{Gruenberg:1970}, p.\,32). Thus every element $u\in \mathfrak f$ can be written uniquely as
$$u=\sum_{\,{_x}u\in \mathbb Z[F],\,x\in
X}\hspace{-.65cm}{_x}u(x-1)=\sum_{x\in X,\,u_x\in \mathbb Z[F]}\hspace{-.55cm}(x-1)u_x.$$  We  refer to $_xu$ (resp. $u_x$) as the {\it left} (resp. {\it right}) {\it partial derivative} of $u$ with respect to the generator $x$.
\par\vspace{.25cm}
Let
$F$ be  a free group of finite rank with basis $\{x_1,\,\ldots\,,\,x_m\}$, $e_1,\,e_2,\,\ldots\,,\,e_m$ integers $\geq 0$ satisfying  $e_m|e_{m-1}|\ldots|e_1$, and $S=\langle x_1^{e_1},\,\ldots\,,\,x_m^{e_m},\, \gamma_2(F)\rangle$. For a two-sided ideal $\mathfrak a$ of the group ring $\mathbb Z[F]$, set  $$D(n,\,\mathfrak a):=F\cap (1+\mathfrak a+\mathfrak f^n).$$  For $1\leq i\leq m$, let
$$
t(x_i,\,e_i):=\begin{cases} 1+x_i+\dots+x_i^{e_i-1},\ \text{if}\ e_i\geq 1,\\
0,\ \text{if}\ e_i=0\end{cases}
$$
Recall the following result of Narain Gupta (\cite{Gupta:87}, Theorem 3.2, p.\,81).\par\vspace{.25cm}
\begin{theorem}
For all $n\geq 1$, modulo $[\gamma_2(F),\,S]\gamma_{n+2}(F),$ the group
$D(n+2,\,\mathfrak f\mathfrak s)$ is generated by
the elements
$$
[x_i,\,x_j]^{t(x_i,\,e_i)a_{ij}},\ 1\leq i<j\leq m,
$$
where $a_{ij}=a_{ij}(x_j,\,\dots\,,\, x_m)\in \mathbb Z[F]$ and
$$
t(x_i,\,e_i)a_{ij}\in t(x_j,e_j)\mathbb Z[F]+\mathfrak s\mathbb Z[F]+\mathfrak
f^n.
$$
\end{theorem}\par\vspace{.25cm}The cases $n=1,\, 2$ of the above Theorem yield the following identification.
\par\vspace{.25cm}
\begin{theorem}\label{fs4}$(i)$ Modulo $\gamma_2(S)\gamma_3(F)$,\par\vspace{.25cm}\noindent$$D(3,\,\mathfrak f\mathfrak s)=\langle [x_i^{e_i},\,x_j]\,|\,1\leq i< j\leq m\rangle.$$
\par\vspace{.25cm}\noindent
$(ii)$ Modulo $\gamma_2(S)\gamma _4(F)$,
$$D(4,\,\mathfrak f\mathfrak s)=\\
\left\langle [x_i^{e_i},\,x_j]^{a_{ij}}\,|\, 1\leq i<j\leq m,\ e_j|\frac{e_i}{e_j}\binom{e_j}{2}a_{ij}\right\rangle.$$
\end{theorem} \proof
We sketch  the proof for (ii), the computations for (i) are similar and simpler. \par\vspace{.25cm}
It is straight-forward to check that all elements  $[x_i^{e_i},\,x_j]^{a_{ij}}$ with \linebreak $ 1\leq i<j\leq m,\ e_j|\frac{e_i}{e_j}\binom{e_j}{2}a_{ij}$, and the subgroup  $\gamma_2(S)\gamma_4(F)$, are contained in $D(4,\,\mathfrak f\mathfrak s)$.\par\vspace{.25cm}
Conversely, let  $w\in D(4,\,\mathfrak f\mathfrak s)$. Then, observe that  modulo $\gamma_2(S)\gamma_4(F)$,  $w=\prod_{i=1}^m w_i $, where \begin{equation}w_i=\prod_{j=i+1}^m [x_i,\,x_j]^{d_{ij}}\in D(4,\,\mathfrak f\mathfrak s),\end{equation} with $d_{ij}=d_{ij}(x_i,\,\ldots\,,\,x_m)$.
From eq.(1) it follows that \begin{equation} d_{ij}\in t(x_i,\,e_i)\mathbb Z[F]+\mathfrak f^4+\mathfrak a,\end{equation}where $\mathfrak a=\mathbb Z[F](\gamma_2(F)-1).$
Thus we have, modulo $\gamma_4(F)\gamma_2(S)$, \begin{equation}
w_i=\prod_{j=i+1}^m[x_i^{e_i},\,x_j]^{a_{ij}}\in D(4,\,\mathfrak f\mathfrak s),\end{equation}where $a_{ij}=a_{ij}(x_i,\ldots,x_m)\in \mathbb Z[F].$
\par\vspace{.5cm}\noindent
Eq.(3) is equivalent to saying that \begin{equation}
t(x_i,\,e_i)\sum_{j=i+1}^m(x_j-1)a_{ij}\in \mathfrak f^3+\mathfrak s.\end{equation}Recall that for $i<j, \ e_j|e_i$. Modulo $\mathfrak f^3+\mathfrak s, $
\begin{multline}
t(x_i,\,e_i)\sum_{j=i+1}^m(x_j-1)a_{ij}=\{{e_i}+\binom{e_i}{2}(x_i-1)\}\sum_{j=i+1}^m(x_j-1)a_{ij}\\=\sum_{j=i+1}^ma_{ij}\frac{e_i}{e_j}e_j(x_j-1)+
\binom{e_i}{2}(x_i-1)\sum_{j=i+1}^m(x_j-1)a_{ij}\\
=-\sum_{j=i+1}^ma_{ij}\frac{e_i}{e_j}\binom{e_j}{2}(x_j-1)^2+
\binom{e_i}{2}(x_i-1)\sum_{j=i+1}^m(x_j-1)a_{ij}
.\end{multline}
Since all  terms in the last expression are in $\mathfrak f^2$, we can assume that the elements $a_{ij}\in \mathbb Z$.
\par\vspace{.5cm}\noindent
Since $\operatorname{SP}^2(F/S)\simeq \mathfrak f^2/(\mathfrak f^3+\mathfrak s)$, eq.(5) holds if and only if \begin{equation}e_j|\frac{e_i}{e_j}\binom{e_j}{2}a_{ij},\quad
e_j|\binom {e_i}{2}a_{ij}.\end{equation}Note that the latter divisibility requirement  holds if the former does.  Thus we have the claimed  identification. $\Box$ \par\vspace{.5cm}
Since $D(4,\,\mathfrak f\mathfrak s)\leq D(3,\,\mathfrak f\mathfrak s) $ and $\gamma_4(F)\leq \gamma_3(F)$, we have a  natural map

$$
\theta: D(4,\mathfrak f\mathfrak s)/\gamma_2(S)\gamma_4(F)\to
D(3,\mathfrak f\mathfrak s)/\gamma_2(S)\gamma_3(F).
$$
\par\vspace{.25cm}\noindent
We claim that $\theta $ is a monomorphism. To this end, we need the following
\par\vspace{.25cm}
\begin{lemma}\label{L1}
 $\gamma_3(F)\cap D(4,\,\mathfrak f\mathfrak s)\leq \gamma_4(F)\gamma_2(S)$.
\end{lemma}\par\vspace{.15cm}\proof
Let $w\in \gamma_3(F)\cap D(4,\,\mathfrak f\mathfrak s)$.
Then, modulo $\gamma_4(F)$,  $w=\prod_{i=1}^m w_i$, $$w_i=\prod_{j>i\leq k}[[x_j,\,x_i],\,x_k]^{a_{ijk}}\in D(4,\,\mathfrak f\mathfrak s), \ a_{ijk}\in \mathbb Z.$$ We then have $$\sum_{j>i\leq k}a_{ijk}[x_j-1)(x_i-1)-(x_i-1)(x_j-1)](x_k-1)\in \mathfrak f^4+\mathfrak f\mathfrak s.$$ Comparing the right partial derivatives with respect to  $x_j$, we have, for each $j>i$,
$$\sum_{k\geq i}a_{ijk}(x_i-1)(x_k-1)\in \mathfrak f^3+\mathfrak s\mathbb Z[F].$$ This is possible if and only if $e_k|a_{ijk}$ for each $k\geq i$. But then $[[x_j,\,x_i],\,x_k]^{a_ijk}\in \gamma_4(F)\gamma_2(S)$, and the assertion stands proved. $\Box$

\par\vspace{.25cm}
\begin{corollary}
The natural map $$\theta: D(4,\,\mathfrak f\mathfrak s)/\gamma_2(S)\gamma_4(F)\to D(3,\mathfrak f\mathfrak s)/\gamma_2(S)\gamma_3(F)$$is a monomorphism, and the cokernel of  $\theta$ is an elementary abelian 2-group. \end{corollary}\par\vspace{.25cm}
\proof
Suppose  $w\in D(4,\,\mathfrak f\mathfrak s)\cap (\gamma_2(S)\gamma_3(F))$. Then $w=uv,\ u\in \gamma_2(S),\ v\in \gamma_3(F)$. Since $u\in D(4,\,\mathfrak f\mathfrak s)$, it follows that $v\in \gamma_3(F)\cap D(4,\,\mathfrak f\mathfrak s)$, and therefore, by Lemma \ref{L1}, $v\in \gamma_4(F)\gamma_2(S).$ Hence $\theta $ is a monomorphism.

\par\vspace{.25cm}
The subgroup  $D(3,\,\mathfrak f\mathfrak s)$ is generated, modulo $\gamma_3(F)\gamma_2(S)$, by the elements\linebreak  $[x_j,\,x_i^{e_i}], \ j>i$. From the identification of $D(4,\,\mathfrak f\mathfrak s)$ (Theorem \ref{fs4}), note that $[x_j,\,x_i^{e_i}]^2\in D(4,\,\mathfrak f\mathfrak s)$. It thus follows  that coker $\theta$, which is isomorphic to  the quotient $$D(3,\,\mathfrak f\mathfrak s)/(D(4,\,\mathfrak f\mathfrak s)\gamma_3(F),$$ is an elementary abelian 2-group. $\Box$

\par\vspace{.25cm}
\begin{theorem}\label{typical} If  $F$ is a free group with ordered basis $\{x_1<\dots<x_m\},\linebreak  S=\langle x_1^{e_1},\,\ldots,\,x^{e_m},\,\gamma_2(F)\rangle$, $e_m|e_{m-1}|\ldots|e_1$, and  $R$ is a normal subgroup of $F$ satisfying $R\gamma_2(F)=S$, then the following identifications hold:
\par\vspace{.25cm}\noindent
$(i)$ $D(4,\,\mathfrak f\mathfrak r\mathfrak f))$ is the subgroup generated, modulo $\gamma_4(F)$,  by the elements of the type \begin{itemize}
\item
$[[x_j,\,x_i],\,x_k]^{e_i},  \ i,\,j,\,k\in \{1,\,\ldots\,,\,m\}, \ j>i;$
\item
$[[x_j,\,x_i],\,x_j]^e,\ j>i,\ e_j|e,\ e_i|2e.$
\end{itemize}
\par\vspace{.25cm}\noindent
$(ii)$ $D(4,\,\mathfrak r\mathfrak f\mathfrak r)$ is the subgroup generated, modulo $\gamma_4(F)$,  by the elements of the type \begin{itemize}
\item
$[[x_j,\,x_i],\,x_k]^{e_ie_k},  \ i,\,j,\,k\in \{1,\,\ldots\,,\,m\}, \ j>i;$
\item
$[[x_j,\,x_i],\,x_i]^e,\ j>i,\ (e_je_i)|e,\ e^2_i|2e.$\end{itemize}
\par\vspace{.25cm}\noindent
$(iii)$ $D(4, \,\mathfrak r^2 \mathfrak f+\mathfrak r \mathfrak f \mathfrak r + \mathfrak f \mathfrak r^2) = [[F, R], R]\gamma_4(F)$.
\par\vspace{.25cm}\noindent
$(iv)$ $D(4, \,\mathfrak f^2 \mathfrak r + \mathfrak r \mathfrak f^2)$ = the subgroup generated, modulo $\gamma_4(F)$,  by the elements $[[x_j, \,x_i],\, x_k]^e$, where
$j > i \leq k,  \,e_k | e$, if $k \neq i$,  and when $k = i$, then $e_ j | e$, and $e_i |2e.$
\end{theorem}\par\vspace{.25cm}\proof (i)
Note that
$D(4,\,\mathfrak f\mathfrak r\mathfrak f)=D(4,\,\mathfrak f\mathfrak s\mathfrak f)\leq \gamma_3(F).$
We first check that for $i,\,j,\,k\in \{1,\,\ldots\,,\, m\}$, the simple commutator \begin{equation}\label{frf}[[x_j,\,x_i],\,x_k]^{e_i}\in D(4,\,\mathfrak f\mathfrak s\mathfrak f),\  \text{provided}\ j>i.\end{equation} Modulo $\mathfrak f^4$, we have \begin{multline}\label{frf1}[[x_j,\,x_i],\,x_k]^{e_i}-1=e_i\{[x_j-1)(x_i-1)-(x_i-1)(x_j-1)](x_k-1)\\-(x_k-1)[(x_j-1)(x_i-1)-(x_i-1)(x_j-1)]\}.\end{multline}Since $j>i$, both $e_i(x_i-1)$ and $e_i(x_j-1)$ lie in $\mathfrak s\mathbb Z[F]+\mathfrak f^2$, and consequently (\ref{frf}) holds. Next, consider $[[x_j,\,x_i[,\,x_j]$,  $j>i$, and let $e$ be an integer such that $e_j|e,\ e_i|2e$. Then, modulo $\mathfrak f^4$, we have
\begin{multline}\label{frf14}
[[x_j,\,x_i],\,x_j]^e-1=\\e\{[x_j-1)(x_i-1)-(x_i-1)(x_j-1)](x_j-1)\\-(x_j-1)[(x_j-1)(x_i-1)-(x_i-1)(x_j-1)]\}.
\end{multline}Given $e_i|2e$ and $e_j|e$, it follows that the above expression lies in $\mathfrak f^4+\mathfrak f\mathfrak s\mathfrak f$.

\par \vspace{.5cm}
Let $X$ be the subgroup of $F$ generated by the elements of the type $[[x_j,\,x_i],\,x_k]^{e_i},$ $ i,\,j,\,k\in \{1,\ \ldots\ ,\ m\}, \  j>i$, and $[[x_j,\,x_i],\,x_j]^e,\ j>i,\ e_j|e,\ e_i|2e.$\par\vspace{.5cm} Recall that the quotient $\gamma_3(F)/\gamma_4(F)$ is a free abelian group with basis $$\{[[x_j,\,x_i],\,x_k]\gamma_4(F)\,|\, 1\,\leq \,i,\,j,\,k\,\leq\, m, \  j>i\,\leq \,k\}.$$
Let $w\in D(4,\,\mathfrak f\mathfrak s\mathfrak f)$. Then, modulo $\gamma_4(F)$,  $w=\prod_{i=1}^{m-1} w_i$ with $$w_i=\prod_{j>i\leq k} [[x_j,\,x_i],\,x_k]^{a_{ijk}},\ a_{ijk}\in \mathbb Z.$$By descending induction on $i$, it follows that $w_i\in D(4,\,\mathfrak f\mathfrak s\mathfrak f) $ for each \linebreak $i=1,\,2,\,\ldots,\,m-1$. Now,  modulo $ \mathfrak f\mathfrak s\mathfrak f+\mathfrak f^4$, \begin{multline}\label{frf2}W_i:=w_i-1=\sum_{j>i\leq k}a_{ijk}\{[(x_j-1)(x_i-1)-(x_i-1)(x_j-1)](x_k-1)\\ -(x_k-1)[(x_j-1)(x_i-1)-(x_i-1)(x_j-1)]\}.\end{multline}
Considering the left partial derivative of $W_i$ with respect to $x_i$, we have
\begin{multline}\label{frf3}{_{x_i}}W_i=
\sum_{j>i}a_{iji}[(x_j-1)(x_i-1)-2(x_i-1)(x_j-1)]\\-\sum_{j>i,\,k>i}a_{ijk}(x_k-1)(x_j-1)\in \mathfrak f^3+\mathfrak f\mathfrak s.
\end{multline}\par\vspace{.25cm}\noindent
Considering the right partial derivative of ${_{x_i}}W_i$ with respect to $x_i$, we have
\begin{equation}\label{frf4}
2a_{iji}(x_j-1)\in \mathfrak f^2+\mathfrak s\mathbb Z[F],
\end{equation}and, therefore
\begin{equation}\label{frf7}
\sum_{j>i}a_{iji}(x_j-1)(x_i-1)-\sum_{j>i,\,k>i}a_{ijk}(x_k-1)(x_j-1)\in \mathfrak f^3+\mathfrak f\mathfrak s.
\end{equation}Since the second term in (\ref{frf7})  does not contain $x_i$, it  follows that
\begin{equation}\label{frf5}
\sum_{j>i}a_{iji}(x_j-1)(x_i-1)\in  \mathfrak f^3+\mathfrak f\mathfrak s\end{equation} and
\begin{equation}\label{frf6}\sum_{j>i,\,k>i}a_{ijk}(x_k-1)(x_j-1)\in \mathfrak f^3+\mathfrak f\mathfrak s.\end{equation} From the inclusion (\ref{frf5}), it follows that  $\sum_{j>i}a_{iji}(x_i-1)\in \mathfrak f^2+\mathfrak s\mathbb Z[F]$, and \linebreak consequently
\begin{equation}\label{frf8}e_i|a_{iji}\ \text{for all}\  j>i.\end{equation}It thus follows that \begin{equation}\label{frf13}
[[x_j,\,x_i],\,x_i]^{a_{iji}}\in X\gamma_4(F).
\end{equation}
On taking right partial derivative with respect to $x_k,\ k>i$, the inclusion (\ref{frf6}) shows that we have
\begin{equation}
\sum_{j>i}a_{ijk}(x_j-1)\in \mathfrak f^2+\mathfrak s\mathbb Z[F],
\end{equation}implying that
\begin{equation}\label{frf9}
e_j|a_{ijk}\ \text{for} \ j>i,\,k>i.
\end{equation}
Plugging the foregoing  divisibility conditions, obtained so far, back in (\ref{frf2}), we see that \begin{multline}\label{frf10}
\sum_{j>i<k,\,j\neq k }(a_{ijk}+a_{ikj})[(x_j-1)(x_i-1)(x_k-1)+(x_k-1)(x_i-1)(x_j-1)]\\+\sum_{j>i}2a_{ijj}[(x_j-1)(x_i-1)(x_j-1)]\in \mathfrak f^4+\mathfrak f\mathfrak s\mathfrak f.
\end{multline}Considering the left partial derivative with respect to $x_k$,   we have
\begin{equation}\label{frf11}
\sum_{j>i}(a_{ijk}+a_{ikj})(x_j-1)(x_i-1)\in \mathfrak f^3+\mathfrak f\mathfrak s.
\end{equation} It thus follows that
\begin{equation}
e_i|(a_{ijk}+a_{ikj})\ \text{for all}\ j>i,\,k>i.
\end{equation}The inclusion (\ref{frf10}) thus reduces to \begin{equation}
\sum_{j>i}2a_{ijj}(x_j-1)(x_i-1)(x_j-1)\in \mathfrak f^4+\mathfrak f\mathfrak s\mathfrak f,
\end{equation} implying that
\begin{equation}
e_i|2a_{ijj}\ \text{for all}\ j>i.
\end{equation}Consequently \begin{equation}\label{frf15}
[[x_j,\,x_i],\,x_j]^{a_{ijj}}\in X\gamma_4(F).
\end{equation}
For $j>i\leq k,  \ j\neq k$, we have $e_i|(a_{ijk}+a_{ikj})$, therefore, modulo $X\gamma_4(F)$, \begin{multline}\label{frf12}
[[x_j,\,x_i],\,x_k]^{a_{ijk}}.[[x_k,\,x_i],\,x_j]^{a_{ikj}}\\=[[x_j,\,x_i],\,x_k]^{a_{ijk}}.[[x_k,\,x_i],\,x_j]^{-a_{ijk}}\\=[[x_k,\,x_j],\,x_i]^{-a_{ijk}}=1,\end{multline} provided $k>j.$ Here the first equality holds because $e_i|(a_{ijk}+a_{ikj})$, and\linebreak  $[[x_k,x_i],x_j]^{e_i}\in X\gamma_4(F)$, while the second equality holds in view of Hall-Witt identity, and the last equality holds because $e_j|a_{ijk}$. Similar argument shows that the same conclusion holds if $j>k$.
\par\vspace{.25cm}The foregoing analysis shows that $w_i\in X\gamma_4(F)$, and thus  the proof of (i) is complete.

\par\vspace{,5cm}
\noindent The proofs of (ii)-(iv) are similar and  so  we omit the details.  $\Box$\par\vspace{.5cm}
The preceding theorem identifies a set of generators for $D(4,\,\mathfrak f\mathfrak r\mathfrak f)/\gamma_4(F) $
in terms of the basic commutators $[[x_j,\,x_i],\,x_k], \
j>i,\,i\leq k$ of weight three.  An easy calculation on these
generators shows that another identification can be given as
follows:
\par\vspace{.25cm}
\begin{theorem}\label{frfth}
If  $F$ is a free group with ordered basis $\{x_1<\dots<x_m\},\ \
S=\langle x_1^{e_1},\,\ldots,\,x^{e_m},\,\gamma_2(F)\rangle$,
$e_m|e_{m-1}|\ldots|e_1$, and  $R$ is a normal subgroup of $F$ satisfying $R\gamma_2(F)=S$, then
the  subgroup
$D(4,\,\mathfrak f\mathfrak r\mathfrak f)$ is generated  by the set $X$  consisting of $\gamma_4(F)$ and the   elements of the type
$[[x_j,\,x_i],\,x_k]^e,\  j>i,\ 1\leq k\leq m$ satisfying\\
\begin{quote} \begin{itemize}
\item $e_i\,|\,e,\ \text{if}\ k\neq j$ \item $e_j\,|\,e,\
e_i\,|\,2e,\ \text{k=j}$

\end{itemize}\end{quote}
\end{theorem}\
\par\vspace{.25cm}
Clearly the subgroup $[[R,\,R],\,F]\subset 1+\mathfrak f\mathfrak r\mathfrak f$, and is generated, modulo $\gamma_4(F)$, by the set $Y$ consisting of the   elements $$[[x_j,\,x_i],x_k]^{e_ie_j},\ j>i,\ 1\leq k\leq m.$$ Hence we have

\par\vspace{.5cm}
\begin{corollary} \label{der1}
 $D(4, \,\mathfrak f\mathfrak r\mathfrak f)/[[R,\,R],\,F]\gamma_4(F) \cong \langle X\rangle/\langle Y\rangle\gamma_4(F).$
\end{corollary}
\par\vspace{.25cm}
\begin{theorem}
The subgroup $D(4,\,\mathfrak r^2\mathfrak f)$ is generated, modulo $\gamma_4(F)$,  by the elements $[[x_j,\,x_i],\,x_k]^{\ell_{ijk}},\ 1\leq i,\,j,\,k\leq m,\ j>i\leq k,$ where \begin{quote}$ \ell_{ijk}=\begin{cases}\operatorname{lcm}(e_ie_j,\,e_ie_k,\,e_je_k),\ k\neq i\\
e_ie_j,\ k=i\end{cases}.$\end{quote}
\end{theorem}\par\vspace{.25cm}
\proof As in the proof of the previous result, we have $$D(4,\,\mathfrak r^2\mathfrak f)=D(4,\,\mathfrak s^2\mathfrak f)\leq \gamma_3(F).$$It is easy to see that the triple commutators occurring in the statement of the theorem  lie in $D(4,\,\mathfrak s^2\mathfrak f)$. Further, if $w\in D(4,\,\mathfrak s^2\mathfrak f)$, then $$w=\prod_{i=1}^{m-1}w_i,\quad w_i=[[x_j,\,x_i],x_k]^{a_{ijk}}, \ j>i\leq k,\ w_i\in D(4,\,\mathfrak s^2\mathfrak f).$$
Modulo $\mathfrak s^2\mathfrak f+\mathfrak f^4$, we  have
\begin{multline}\label{221}w_i-1=\sum_{j>i\leq k}a_{ijk}\{[(x_j-1)(x_i-1)-(x_i-1)(x_j-1)](x_k-1)\\ -(x_k-1)[(x_j-1)(x_i-1)-(x_i-1)(x_j-1)]\}.\end{multline}
Considering the left partial derivative with respect to  $x_i$, we have
\begin{multline}\label{231}
\sum_{j>i}a_{iji}\{(x_j-1)(x_i-1)-2(x_i-1)(x_j-1)\}\\-\sum_{j>i<k}a_{ijk}(x_k-1)(x_j-1)\in \mathfrak s^2+\mathfrak f^3.
\end{multline}
Observe that $\mathfrak f^2/\mathfrak f^3$ is a free abelian group with basis $$(x_i-1)(x_j-1)+\mathfrak f^3, \ 1\leq i, \,j\leq m,$$ and $\mathfrak s^2+\mathfrak f^3$ is generated by the elements $$e_ie_j(x_i-1)(x_j-1)+\mathfrak f^3, 1\leq i, \,j\leq m.$$ Thus  the inclusion (\ref{231}) yields
\begin{equation}
e_ie_j|a_{iji},\quad e_je_k|a_{ijk}\ \text{for}\ j>i<k,
\end{equation}and in turn \begin{equation}
\sum_{j>i<k}a_{ijk}(x_k-1)(x_j-1)\in \mathfrak s^2+\mathfrak f^3.\end{equation}
The above inclusion then yields \begin{equation}
e_je_k|a_{ijk}\ \text{for}\ j>i<k.
\end{equation}
The congruence (\ref{221}) now implies that  \begin{multline}
\sum_{j>i< k}a_{ijk}\{[(x_j-1)(x_i-1)-(x_i-1)(x_j-1)](x_k-1)\\ +(x_k-1)(x_i-1)(x_j-1)]\}\in \mathfrak s^2\mathfrak f+\mathfrak f^4.
\end{multline}
Considering the left partial derivative with respect to  $x_j$, we obtain \begin{equation}\label{271}
a_{ijj}[(x_j-1)(x_i-1)-(x_i-1)(x_j-1)]+\sum_{j>i< k}a_{ijk}(x_k-1)(x_i-1)\in \mathfrak s^2+\mathfrak f^3.
\end{equation}
Thus we have \begin{equation}
e_ie_j|a_{ijj},\ e_ie_k|a_{ijk}\ \text{for}\ j>i<k,
\end{equation}forcing (\ref{271}) to yield \begin{equation}
e_ie_j|a_{ijk}, \text{for}\ j>i<k.
\end{equation}
The various divisibility conditions derived above for the integers $a_{ijk} $ clearly suffice for the claimed assertion. $\Box$
\par\vspace{.25cm}
\begin{theorem}
$(i)$ $D(3,\,\mathfrak s^2)\equiv \gamma_2(S)\mod \gamma_3(F).$
$(ii)$
$D(4,\,\mathfrak s^2\mathbb Z[F])\equiv \gamma_2(S)\mod \gamma_4(F).$
\end{theorem}
\proof Clearly $\gamma_2(S)\leq 1+\mathfrak s^2$.\par\vspace{.25cm}\noindent
(i) Let $w\in D(3,\,\mathfrak s^2)$. Then, modulo $\gamma_3(F)$,  $w=\prod_{j>i}[x_j,\,x_i]^{a_{ij}}$, $a_{ij}\in \mathbb Z$, with $$\sum_{j>i}[(x_j-1)(x_i-1)-(x_i-1)(x_j-1)]a_{ij}\in \mathfrak f^3+\mathfrak s^2.$$Comparing the right partial derivative with respect to $x_i$, we have $$\sum_j(x_j-1)a_{ij}\in \mathfrak f^2+e_i\mathfrak s.$$Comparing the right partial derivative with respect to $x_j$, we conclude that $$e_ie_j\,|\, a_{ij}.$$ Hence it follows that $$\prod_{j>i}[x_j,\,x_i]^{a_{ij}}=\prod_{j>i}[x_j^{e_j},\,x_i^{e_i}]^{b_{ij}},\ b_{ij}\in \mathbb Z,$$ and consequently $w\in \gamma_3(F)\gamma_2(S)$.
\par\vspace{.5cm}\noindent
(ii) Let $w\in D(4,\,\mathfrak s^2\mathbb Z[F])$. By part (i), we have $$w=w_0\mod \gamma_2(S)$$ for some $w_0\in \gamma_3(F). $ We claim that $w_0\in \gamma_4(F)\gamma_2(S). $ Now, modulo $\gamma_4(F)$, $$w_0=\prod_iw_i,\ w_i=\prod_{j>i\leq k}[[x_j,\,x_i],\,x_k]^{a_{ijk}},\ a_{ijk}\in \mathbb Z,$$ and, as can be seen by descending induction on $i$, we have,  modulo $\mathfrak f^4+\mathfrak s^2\mathbb Z[F],$ \begin{multline}w_i-1=\sum_{j>i\leq k}a_{ijk}\{(x_k-1)[(x_j-1)(x_i-1)-(x_i-1)(x_j-1)]\\-[(x_j-1)(x_i-1)-(x_i-1)(x_j-1)](x_k-1)\}.\end{multline}
\par\vspace{.25cm}\noindent Comparing the right partial derivative with respect to $x_i$, we have
\begin{multline}
\sum_{j>i}a_{iji}[(x_j-1)(x_i-1)-(x_i-1)(x_j-1)]\\+\sum_{j,\,k}(x_j-1(x_k-1)a_{ijk}\in \mathfrak f^3+t(x_i,\,e_i)\mathfrak s\mathbb Z[F]+\mathfrak f\mathfrak s.\end{multline}Next, comparing right partial derivative with respect to $x_j$, we have $$a_{iji}(x_i-1)+\sum_k(x_k-1)a_{ijk}\in \mathfrak f^2+e_it(x_j,\,e_j)\mathbb Z[F]+e_i\mathfrak f+\mathfrak s\mathbb Z[F].$$Thus it follows that \begin{equation}\label{40}
e_k\,|\, a_{ijk}, \ k\neq i,\end{equation} \par\vspace{.25cm}\noindent
The condition (\ref{40}) implies that
$$w_i=w_{i1}w_{i2},\ w_{i1}=\prod_{j>i}[[x_j,\,x_i],\,x_i]^{a_{iji}},\ w_{i2}=\prod_{j>i<k}[[x_j,\,x_i],\,x_k]^{a_{ijl}}$$ with
$w_{i2}\in \gamma_2(S)\gamma_4(F).$ Consequently, we have
$$w_{i1}-1\in \mathfrak f^4+\mathfrak s^2\mathbb Z[F], $$ i.e.,
\begin{multline}
\sum_{j>i}a_{iji}\{(x_i-1)[(x_j-1)(x_i-1)-(x_i-1)(x_j-1)]\\-[(x_j-1)(x_i-1)-(x_i-1)(x_j-1)](x_i-1)\}\in \mathfrak f^4+\mathfrak s^2\mathbb Z[F].\end{multline}
Comparing right partial derivative with respect to $x_j$, we have
$$a_{iji}(x_i-1)^2\in \mathfrak f^3+t(x_j,\,e_j)\mathfrak s+\mathfrak f\mathfrak s.$$ Next comparing right partial derivative with respect to $x_i$, we have $$(x_i-1)a_{iji}\in \mathfrak f^2+e_jt(x_i,\,e_j)\mathbb Z[F]+e_i\mathfrak f+\mathfrak s\mathbb Z[F].$$Thus it follows that $e_i\,|\, a_{iji}.$ Hence $w_{i1}\in \gamma_2(S)\gamma_4(F)$. $\Box$\par\vspace{.5cm}On observing that $\gamma_2(S)\leq (\gamma_2(R)\gamma_3(F))\cap (R\gamma_4(F))$, we have the following result of which (ii) answers Narain Gupta's Problem 6.9(a) in \cite{Gupta:87}.\par\vspace{.25cm}
\begin{corollary}\label{Guptaprob}$(i)$ $D(3,\,\mathfrak r^2)=\gamma_2(R)\gamma_3(F)$\quad and \quad $(ii)$
$D(4,\ \mathfrak r^2\mathbb Z[F])\leq R\gamma_4(F).$
\end{corollary}\par\vspace{.5cm}
 If $\mathcal A$ is a nilpotent ring of exponent $n+1$, i.e., $\mathcal A^{n+1}=(0)$, then under the circle operation $$a\circ b=a+b+ab,\ a,\,b\in \mathcal A,$$we have a group $(\mathcal A,\,\circ)$, called the {\it circle group} of $\mathcal A$.
Suppose we have a homomorphism $\theta:G\to (\mathcal A,\,\circ)$. Then we can extend $\theta$, by linearity, to a ring homomorphism $$\theta^*:\mathbb Z[G]\to R.$$It is then clear that $\theta^*(\mathfrak g^i)\leq \mathcal A^i$, and, in particular, $$\theta^*(\mathfrak g^{n+1})=(0). $$Thus we have the following: \par\vspace{.25cm}
\begin{proposition} $(Sandling$ \cite{Sandling:1974}$)$ If a group $G$ embeds in the circle group of a nilpotent ring of exponent ${n+1}$, then the $(n+1)$th dimension subgroup $D_{n+1}(G)=\{1\}$.
\end{proposition}
\par
As a consequence of the above Proposition and the existence of nilpotent groups of class three without dimension property, we note that, {\it in general, a nilpotent group of class three cannot be embedded in the circle group of a nilpotent ring of exponent four} (see \cite{Ault-Watters}).
\Para
In the foregoing discussion, we have identified the subgroups $$D(4,\,\mathfrak a)=F\cap (1+\mathfrak a+\mathfrak f^4)$$ for various two-sided  ideals $\mathfrak a\subset \mathfrak f$. In each of these cases $F/D(4,\,\mathfrak a)$ embeds in the circle group of the nilpotent ring $\mathfrak f/(\mathfrak a+\mathfrak f^4)$ of exponent four under the map $$fD(4,\,\mathfrak a)\mapsto f-1+\mathfrak a+\mathfrak f^4.$$ Thus in each of these cases we are getting  a nilpotent group of class three which has the {\it dimension property}, i.e., a group whose dimension series agrees with its lower central series.

\par\vspace{.5cm}
\section{Derived  functors}Let us recall the construction of derived
functors in the sense of Dold-Puppe \cite{DoldPuppe}. For an endofunctor $F$ on the category $\mathfrak A$ of
abelian groups, the bigraded sequence $L_i(-,\,n),\  i\geq 0,\ n\geq 0$,   of derived functors of $F$ is defined by setting
$$
L_iF(A,\,n)=\pi_i(FKP_\ast[n]),\  A\in \mathfrak A,
$$
where $P_*[n]_\ast \to A$ is a projective resolution of $A$ starting at level $n$, and
 $K$ is  the Dold-Kan transform,  inverse to the Moore normalization
 functor from simplicial abelian groups to chain complexes. We are  concerned only with the case when $n=0$, and, as such, we drop the index $n$ and write $L_iF(A):=L_i(A,\,0)$. Our aim in this section is to explore relationship between generalized dimension subgroups and derived functors.  \par\vspace{.25cm}
Recall that, if $\operatorname{SP}^2(A)$ denotes the second symmetric power of the abelian group $A$, then  $A\mapsto \operatorname{SP}^2(A)$ is a non-additive functor of degree 2 on the category $\mathfrak A$. As observed in  \cite{HMP}, the functor $L_1\operatorname{SP}^2$ is related to the generalized dimension subgroup  $D(3,\,-)$.
Let us give another  proof of that relationship.
\par\vspace{.25cm}
\begin{theorem}\label{D(3)} Let $G$ be a group, $G_{ab}$ its abelianization $G/\gamma_2(G)$,  and $F/R$  a free presentation of  $G$. Then
\begin{equation}\label{s2formula}
L_1\operatorname{SP}^2(G_{ab})\cong D(3,\,\mathfrak r\mathfrak f)/\gamma_2(R)\gamma_3(F).
\end{equation}\end{theorem}\par\vspace{.25cm}\proof

Set $\bar F=F_{ab},\ \bar R=R\gamma_2(F)/ \gamma_2(F)$. Then $G_{ab}\cong\bar
F/\bar R$. Let $\{x_i\}_{i\in I}$ be a free basis of $F$. The standard identifications
$$\frac{\bar F\otimes \bar F}{\bar R\otimes \bar F}=
\frac{\mathfrak f^2}{\mathfrak r\mathfrak f+\mathfrak
f^3},\quad \frac{\bar F\wedge \bar
F}{\bar R\wedge \bar R}=
\frac{\gamma_2(F)}{\gamma_2(R)\gamma_3(F)}
$$induced by $\bar x_i\otimes \bar x_j\mapsto (x_i-1)(x_j-1)+{\mathfrak r\mathfrak f+\mathfrak
f^3}$, and $\bar x_i\wedge \bar x_j\mapsto [x_i,\,x_j]\gamma_2(R)\gamma_3(F)$ respectively,
imply that there is a natural exact sequence
\begin{equation}\label{3term}
0\to \frac{D(3, \, \mathfrak r\mathfrak f)}{\gamma_2(R)\gamma_3(F)}\to \frac{\bar F \wedge \bar F}{\bar
R\wedge \bar R}\to\frac{\bar F\otimes \bar F}{\bar R\otimes \bar
F}
\end{equation}
Now recall that the Koszul complex
\begin{equation}\label{koszul1} 0\to \bar R\wedge \bar R\to \bar
R\otimes \bar F\to \operatorname{SP}^2(\bar F)
\end{equation}
defines the object $L\operatorname{SP}^2(G_{ab})$ of the derived category of
abelian groups; therefore, $H_1$ of the complex (\ref{koszul1}) is
naturally isomorphic to $L_1\operatorname{SP}^2(G_{ab})$ (\cite{BreenMikhailov}, \cite{Kock:2001}). The claimed
identification (\ref{s2formula}) follows from the following
commutative diagram with exact columns
$$
\xyma{\bar R\wedge \bar R\ar@{>->}[d]\ar@{>->}[r] & \bar R\otimes
\bar F\ar@{->}[r]\ar@{>->}[d] & \operatorname{SP}^2(\bar F)\ar@{=}[d]\\ \bar
F\wedge \bar F\ar@{>->}[r] \ar@{->>}[d] & \bar F\otimes \bar
F\ar@{->>}[r]\ar@{->>}[d] &\operatorname{SP}^2(\bar F)\\ \frac{\bar F\wedge \bar
F}{\bar R\wedge \bar R}\ar@{->}[r] & \frac{\bar F\otimes \bar
F}{\bar R\otimes \bar F}}
$$
where the middle sequence is the standard Koszul exact sequence
and the lower horizontal map is the map from (\ref{3term}). $\Box$




\par\vspace{.5cm}

Let us next recall the well-known description (see, for example,
MacLane \cite{Mac1}), of the derived functor
$L_1\otimes^2=\operatorname{Tor}^{[2]}:\mathfrak A\to \mathfrak
A$. Given $A\in \mathfrak A$, the group $\operatorname{Tor}^2(A)$
is generated by the $n$-linear expressions $\tau_h(a_1,a_2)$
(where all $a_i$ belong to the subgroup $ {}_hA:=\{a\in A\,|\,
ha=0\},  \ h
>0)$, subject to the so-called slide relations
\begin{equation}
\label{slide} \tau_{hk}(a_1,\,a_2) = \tau_{h}(ka_1, \,a_2)
\end{equation} for all $i$ whenever  $hka_1 = 0$ and $ha_2=0$ and
analogous relation, where the roles of $a_1,\ a_2$ are changed.

The natural map $\otimes^2 \to \operatorname{SP}^2$  induces  a
natural epimorphism
\begin{equation} \label{epi} L_1\otimes^2(A)\to L_1\operatorname{SP}^2(A) \end{equation} which maps
the generator $\tau_h(a_1,\,a_2)$ of $ L_1\otimes^2(A)\cong
\operatorname{Tor}(A,\,A)$ to the generator $\beta_h(a_1,\,a_2)$ of
$L_1\operatorname{SP}^2(A)$ so that the kernel of this map is
generated by the elements $\tau_h(a,\,a),\ a\in\ _hA.$

It is shown by Jean  in  \cite{Jean} that
\begin{equation}\label{syder}
L_1\operatorname{SP}^{3}(A)\simeq
(L_1\operatorname{SP}^{2}(A)\otimes A)/Jac_{S},
\end{equation}
where $Jac_{S}$ is the subgroup generated by elements of the form
$$
\beta_h(x_1,\,x_2)\otimes x_3-\beta_h(x_1,\,x_3)\otimes
x_2+\beta_h(x_2,\,x_3)\otimes x_1
$$
with $x_i \in {}_hA$.
\par\vspace{.25cm}
There is a natural connection between  these functors and the  homology of
Eilenberg-MacLane spaces $K(\Pi,\,n)$ (see, for example, \cite{Breen},
\cite{BreenMikhailov}). In particular, there are natural isomorphisms
$$
H_5K(A,\,2)\cong L_1\Gamma^2(A),\ \ H_7K(A,\,2)\cong L_1\Gamma^3(A)
$$
where $\Gamma^2,\  \Gamma^3$ are divided square and cube functor respectively. The
natural transformations $\operatorname{SP}^2\to \Gamma^2$, $\operatorname{SP}^3\to \Gamma^3$
induce, for $A\in \mathfrak A$,  the following exact sequences (\cite{Jean},
\cite{BreenMikhailov}) (which do not split functorially):
$$
0\to L_1\operatorname{SP}^2(A)\to H_5K(A,\,2)\to \operatorname{Tor}(A,\,\mathbb Z/2\mathbb Z)\to 0
$$
and
\begin{equation}\label{h7}
0\to L_1\operatorname{SP}^3(A)\to H_7K(A,\,2)\to \operatorname{Tor}_1(A,\,A,\,\mathbb Z/2\mathbb Z)\oplus
\operatorname{Tor}(A,\,\mathbb Z/3\mathbb Z)\to 0
\end{equation}\par\vspace{.25cm}We are now ready to present our main result on the relation between generalized dimension subgroups and derived functors.
\par\vspace{.25cm}
\begin{theorem}\label{main1}
If $G_{ab}$ is 2-torsion-free, then $$ \frac{D(4,\,\mathfrak
f\mathfrak r\mathfrak f)}{[R,\,R,\,F]\gamma_4(F)}\simeq
L_1\operatorname{SP}^3(G_{ab})
$$
and there is a natural short exact sequence
\begin{equation}\label{frfsq}
\frac{D(4,\,\mathfrak f\mathfrak r\mathfrak f)}{[R,\,R,\,F]\gamma_4(F)}\hookrightarrow
H_7K(G_{ab},\,2)\twoheadrightarrow \operatorname{Tor}(G_{ab},\,\mathbb Z/3)
\end{equation}
In particular, if $G_{ab}$ is both 2-torsion-free and 3-torsion-free, then there is
a natural isomorphism
\begin{equation}\label{frfis}
\frac{D(4,\,\mathfrak f\mathfrak r\mathfrak f)}{[R,\,R,\,F]\gamma_4(F)}\cong H_7K(G_{ab},\,2).
\end{equation}
\end{theorem}
\par\vspace{.25cm}
\begin{proof} Let us first consider the natural homomorphism
\begin{multline*}
q: L_1\operatorname{SP}^2(G_{ab})\otimes G_{ab}\to \\ Q(F,\,R):=\frac{\langle
[[x_j,\,x_i],\,x_k]^{e_i},\
m\geq j>i\geq 1,\ 1\leq k\leq m\rangle\gamma_4(F)}{[R,\,R,\,F]\gamma_4(F)}
\end{multline*}
defined by
$$
q: \beta_h(\bar x_1,\,\bar x_2)\otimes \bar x_3\mapsto
[x_1,\,x_2,\,x_3]^h.[R,\,R,\,F]\gamma_4(F),\ x_1^h,\ x_2^h\in R\gamma_2(F).
$$
Clearly, $q$ is an epimorphism. For $x_3\in F$ such that $x_3^h\in
R\gamma_2(F)$, observe that \begin{multline} q(\beta_h(\bar x_1, \bar
x_2)\otimes \bar x_3-\beta_h(\bar x_1, \bar x_3)\otimes \bar
x_2+\beta_h(\bar x_2, \bar x_3)\otimes \bar x_1)=\\
[x_1,x_2,x_3]^h[x_1,x_3,x_2]^{-h}[x_2,x_3,x_1]^h.[R,R,F]\gamma_4(F).
\end{multline}
By Hall-Witt identity, this element is trivial in $Q(F,R)$. By
(\ref{syder}), the map $q$  factors through $L_1\operatorname{SP}^3(G_{ab})$.
That is, we obtain an induced epimorphism
$$
q': L_1\operatorname{SP}^3(G_{ab})\to Q(F,R).
$$
The formula for cross-effect of $L_1\operatorname{SP}^3$ is given as follows (see
\cite{BreenMikhailov}, \cite{Jean}): for abelian groups $A,\ B$,
$$
L_1\operatorname{SP}^3(A\oplus B)=L_1\operatorname{SP}^3(A)\oplus L_1\operatorname{SP}^3(B)\oplus L_1\operatorname{SP}^3(A|B),
$$
with a short exact sequence, which splits (unnaturally)
\begin{multline}
0\to (L_1\operatorname{SP}^2(A)\otimes B)\oplus (A\otimes L_1\operatorname{SP}^2(B))\to
L_1\operatorname{SP}^3(A|B)\to\\ \operatorname{Tor}(\operatorname{SP}^2(A),B)\oplus \operatorname{Tor}(A,\,\operatorname{SP}^2(B))\to 0.
\end{multline}
Observe that using the formula for cross-effect of a polynomial functor
together with the value of the  functor on cyclic groups one can
compute the value of the functor on any finitely generated
abelian group. Since $L_1\operatorname{SP}^2(\mathbb Z/l\mathbb Z)=L_1\operatorname{SP}^3(\mathbb Z/l\mathbb Z)=0$
for any $l\geq 0$, we can easily compute $L_1\operatorname{SP}^3(G_{ab})$. It turns out that the
group $L_1\operatorname{SP}^3(G_{ab})$ is freely generated by triples $w(i,\,j,\,k),\
i<j<k$, and $t(i,\,j,\,k),\linebreak  i\leq j<k$, such that the exponent of
$w(i,j,k)$ and of $t(i,\,j,\,k)$ is exactly $e(k)$:
\begin{multline}
L_1\operatorname{SP}^3(G_{ab})\simeq \langle w(i,\,j,\,k), \ i<j<k,\ t(i,\,j,\,k),\ i\leq
j<k\ |\\ e(k)w(i,\,j,\,k)=e(k)t(i,\,j,\,k)=0\rangle
\end{multline}
The map $q'$ then is given as follows:
\begin{align*}
& w(i,j,k)\mapsto [x_j,x_i,x_k]^{e_i},\\
& t(i,j,k)\mapsto [x_k,x_i,x_j]^{e_i}.
\end{align*}
Observe that, as an abelian group, the group $Q(F,\,R)$  is freely
generated by brackets $[x_j,\,x_i,\,x_k]^{e_i},\ i<j$ with relations
$([x_j,\,x_i,\,x_k]^{e_i})^{lcm(e_j,\,e_k)}=1$. Hence $q'$ is an
isomorphism.

Observe that, if $G_{ab}$ is 2-torsion-free, Corollary \ref{der1}
implies that there is a natural isomorphism
$$
Q(F,\,R)\cong \frac{D(4,\,\mathfrak f\mathfrak r\mathfrak
f)}{[R,\,R,\,F]\gamma_4(F)}.
$$
The sequence (\ref{frfsq}) and the isomorphism (\ref{frfis}) now
follow from the isomorphism $Q(F,\,R)\cong L_1\operatorname{SP}^3(G_{ab})$ and
the sequence (\ref{h7}).
\end{proof}

We point out that, Theorem \ref{cubicpicture} (v) from the next
section implies that, for any free presentation $G\cong F/R$, there is a natural embedding
$$
L_1\operatorname{SP}^3(G_{ab})\hookrightarrow\frac{D(4,\,\mathfrak
f\mathfrak r\mathfrak f)}{[R,\,R,\,F]\gamma_4(F)}.
$$
The proof of \ref{cubicpicture} (v) is not combinatorial.

\vspace{.5cm}

For an abelian group $A$ define $\mathfrak L_s^3(A)$ to be the  abelian group
generated by brackets $\{a,\,b,\,c\},\ a,\,b,\,c\in A$ which are additive
on each variable, with the following defining relations:
\begin{align*}
& \{a,b,c\}=\{b,a,c\},\\
& \{a,b,c\}+\{c,a,b\}+\{b,c,a\}=0.
\end{align*}
This construction defines a functor $\mathfrak L_s^3: {\sf Ab}\to {\sf Ab}$,
which we  call {\it the third super Lie functor}. The
difference between $\mathfrak L_s^3$ and the third super-Lie
functor $\EuScript L_s^3$ considered in \cite{BreenMikhailov} is
that, in general, $\{a,a,a\}\neq 0$ in $\mathfrak L_s^3$ and one
can get the functor $\EuScript L_s^3$ as a natural quotient of
$\mathfrak L_s^3$ by brackets of the type $\{a,a,a\}$, and we have the following commutative diagram:
$$
\xyma{& & A\otimes \mathbb Z/3\mathbb Z\ar@{>->}[d] \\ \operatorname{SP}^3(A) \ar@{>->}[d]
\ar@{>->}[r] & \operatorname{SP}^2(A)\otimes
A\ar@{->>}[r] \ar@{>->}[d] & \mathfrak L_s^3(A) \ar@{->>}[d]\\
\Gamma^3(A)\ar@{>->}[r]\ar@{->>}[d] & \Gamma^2(A)\otimes
A\ar@{->>}[r]\ar@{->>}[d] & \EuScript L_s^3(A)\\
A\otimes \mathbb Z/3\mathbb Z\oplus (A\otimes A\otimes \mathbb
Z/2\mathbb Z)\ar@{->>}[r] & A\otimes A\otimes \mathbb Z/2\mathbb Z.}
$$
\par\vspace{.25cm}
\begin{theorem}\label{super} If $G$ is a group and $F/R $ its free presentation, then
$$\frac{D(4,\,\mathfrak r^2\mathfrak f)}{\gamma_3(R)\gamma_4(F)}\cong L_2\mathfrak L_s^3(G_{ab}).$$
\end{theorem}
\begin{proof}
Let us first  construct a natural model for the element $L\mathfrak
L_s^3(G_{ab})$ of the \linebreak derived category of abelian groups. We have the the following commutative diagram:
$$\xyma{\Lambda^3(\bar R) \ar@{>->}[r]\ar@{>->}[d] & \Lambda^2(\bar
R)\otimes
\bar R\ar@{->>}[r] \ar@{>->}[d] & \EuScript L^3(\bar R)\ar@{>->}[d]\\
\Lambda^2(\bar R)\otimes \bar F\ar@{>->}[r] \ar@{->}[d] &
\Lambda^2(\bar R)\otimes \bar F\oplus (\bar R\otimes \bar R\otimes
\bar F) \ar@{->>}[r]\ar@{->}[d] & \bar R\otimes \bar R\otimes \bar F\ar@{->}[d]\\
\bar R\otimes \operatorname{SP}^2(\bar F)\ar@{>->}[r] \ar@{->}[d] & \bar R\otimes
\operatorname{SP}^2(\bar F)\oplus (\bar F\otimes \bar R\otimes \bar F)
\ar@{->>}[r] \ar@{->}[d] & \bar F\otimes \bar R\otimes \bar F\ar@{->}[d]\\
\operatorname{SP}^3(\bar F)\ar@{>->}[r] & \operatorname{SP}^2(\bar F)\otimes \bar F\ar@{->>}[r] &
\mathfrak L_s^3(\bar F)}
$$
where the left vertical complex is the Koszul complex representing
the element $L\operatorname{SP}^3(G_{ab})$, the middle complex is the tensor
product of the Koszul complex with $\bar R\to \bar F$. Also we have the following commutative diagram:

$$
\xyma{\EuScript L^3(\bar R)\ar@{>->}[r] \ar@{>->}[d] & \bar
R\otimes \bar R\otimes \bar F\ar@{>->}[d] \ar@{->}[r] & \bar
R\otimes \bar F\otimes \bar F \ar@{->}[r] \ar@{>->}[d] & \mathfrak L_s^3(\bar F)\ar@{=}[d]\\
\EuScript L^3(\bar F)\ar@{>->}[r] \ar@{->>}[d] & \bar F\otimes
\bar F\otimes \bar F\ar@{->}[r] \ar@{->>}[d] & \bar F\otimes \bar
F\otimes \bar F\ar@{->>}[r] \ar@{->>}[d] & \mathfrak L_s^3(\bar
F)\\
\EuScript L^3(\bar F)/\EuScript L^3(\bar R)\ar@{->}[r] &
\frac{\bar F\otimes \bar F\otimes \bar F}{\bar R\otimes \bar
R\otimes \bar F}\ar@{->}[r] & \frac{\bar F\otimes \bar F\otimes
\bar F}{\bar F\otimes \bar R\otimes \bar F} }
$$
in which the  middle horizontal sequence is exact. Since the upper
horizontal \linebreak sequence models $L\mathfrak L_s^3(G_{ab}),$ the kernel
of the lower left horizontal map is $L_2\mathfrak L_s^3(G_{ab}).$
On invoking the identifications
$$
\frac{\mathfrak f^3}{\mathfrak r^2\mathfrak f+\mathfrak
f^4}=\frac{\bar F\otimes \bar F\otimes \bar F}{\bar R\otimes \bar
R\otimes \bar F},\ \
\gamma_3(F)/(\gamma_3(R)\gamma_4(F))=\EuScript L^3(\bar
F)/\EuScript L^3(\bar R),
$$
our assertion   follows from
the following natural exact sequence
$$
0\to\frac{D(4,\,\mathfrak r^2\mathfrak f)}{\gamma_3(R)\gamma_4(F)}\to \EuScript L^3(\bar F)/\EuScript
L^3(\bar R)\to \frac{\bar F\otimes \bar F\otimes \bar F}{\bar
R\otimes \bar R\otimes \bar F}.
$$
\end{proof}

\vspace{.5cm}

\section{Limits}

Let $\mathfrak A$ be the category of abelian groups. A
{\it representation} of a category $\mathcal C$ is a functor $\mathcal
F:\mathcal C\to \mathfrak A.$  Let $\mathfrak A^{\mathcal C}$ denote the category of representations of
$\mathcal C$.  The diagonal
functor $diag: \mathfrak A\to \mathfrak A^{\mathcal C}$ is the functor that maps an abelian
group $A$ to a constant representation  $\mathcal C\to {\mathfrak A}$, namely, the one which sends all objects from $\mathcal C$ to $A$ and all morphisms to
${\rm id}_A.$ The {\it limit of a representation}  $\ilimit $ is a left exact
functor  $\ilimit : \mathfrak A^{\mathcal
C}\to \mathfrak A$. There are different equivalent ways to define this functor, one way,  for example, is to define it as the right adjoint functor to
the diagonal functor. We refer the reader to the recent paper
\cite{IvanovMikhailov} for the theory of the functor  $\ilimit$  and its derived functors.
If  $\mathcal C$ is a category such that for any two objects $C$
and $C'$ there exists a morphism $f:C\to C'.$ then, for a functor,
$\mathcal F:\mathcal C\to \mathfrak A$,  $\ilimit\ \mathcal
F$ is the largest constant subfunctor of $\mathcal F$ and one can
perceive $\ilimit\ \mathcal F$ as the largest subgroup of
$\mathcal F(C)$ which does not depend on $C$.
\par\vspace{.25cm}Let $\mathcal G$ be the category of groups, and $\mathcal E$ the category of free presentations $$E:\ 1\to R\to F\to G\to 1$$ of $G\in \mathcal G$. Then,  every functor $\mathcal F:\mathcal E\to \mathfrak A$     gives rise to a functor $\mathcal G\to \mathfrak A,\ G\mapsto \ilimit \mathcal F$. It is naturally of interest to examine the limits of such functors. For example, if  $\mathcal F:\mathcal E\to \mathcal G$ is a functor satisfying $\mathcal F(E)\subseteq \gamma_n(F)$, $n\geq 2$, then we have a functor $\bar{\mathcal F}:\mathcal E\to \mathfrak A$, $E\mapsto \gamma_n(F)/\mathcal F(E)\gamma_{n+1}(F)$,  and there arises the  problem of describing $\ilimit \bar{\mathcal F}$. We examine in the present  section instances of this type. Let us first recall an observation from \cite{EmmMikh}.
\par\vspace{.25cm}
Let $\mathcal{C}$ be a category with pairwise coproducts, i.e.,  having the property that for
any two objects $C_1,\ C_2\in \mathcal{C}$ there exists the coproduct
$C_1 \overset{i_1}{\longrightarrow} C_1\sqcup C_2
\overset{i_2}{\longleftarrow} C_2$ in $\mathcal{C}$. Then, for any
representation $\mathcal{F}:\mathcal C\to \mathfrak A$,  there is a natural
transformation
\begin{equation}{\sf T}_{\mathcal F,\,C}:\mathcal{F}(C)\oplus \mathcal{F}(C) \longrightarrow \mathcal{F}(C\sqcup C),
\end{equation}
given by ${\sf T}_{\mathcal
F,\,C}=(\mathcal{F}(i_1),\,\mathcal{F}(i_2)).$ The representation
$\mathcal{F}$ of $\mathcal{C}$ is said to be\linebreak  {\it monoadditive}
(resp. {\it additive}) if ${\sf T}_{\mathcal F,\,C}$  is a
monomorphism (resp. isomorphism) for all $C\in \mathcal C$. The following observation gives a
useful tool for description of  limits (see
\cite{EmmMikh}): \par\vspace{.25cm}\begin{quote}{\it If $\mathcal F:\mathcal C\to \mathfrak A$ is  a monoadditive representation, then  $\ilimit \mathcal F=0$.}\end{quote}
\par\vspace{.25cm}
\begin{theorem}\label{quadraticpicture}{\it (Quadratic functors)}
\begin{align*}
& (i)\ \ \ilimit \frac{\gamma_2(F)}{[R,\,F]\gamma_3(F)}=\Lambda^2(G_{ab})\\
& (ii)\ \ \ilimit \frac{\gamma_2(F)}{(R\cap
\gamma_2(F))\gamma_3(F)}=\gamma_2(G)/\gamma_3(G)\\
& (iii)\ \ \ilimit
\frac{\gamma_2(F)}{\gamma_2(R)\gamma_3(F)}=L_1\operatorname{SP}^2(G_{ab})
\end{align*}
\end{theorem}
\begin{proof}
The cases (i), (ii) are obvious, since the representations in
(i), (ii) are \linebreak constant. To prove the case (iii), consider the exact
sequence
\begin{equation}\label{fcap}
1\to \frac{D(3, \,\mathfrak{rf})}{[R,\,R]\gamma_3(F)}\to \frac{\gamma_2(F)}{[R,\,\,R]\gamma_3(F)}\to
\frac{\mathfrak f^2}{\mathfrak{rf}+\mathfrak f^3}.
\end{equation}
There is a natural isomorphism
$$
\frac{\mathfrak f^2}{\mathfrak{rf}+\mathfrak f^3}\cong
G_{ab}\otimes F_{ab}.
$$
The representation $G_{ab}\otimes F_{ab}$ is monoadditive; moreover, it is, in fact, additive. Therefore, its limit is zero. The
statement (iii) follows from the identification of the left hand
side of (\ref{fcap}) with $L_1\operatorname{SP}^2(G_{ab})$ (see
Theorem \ref{D(3)}).
\end{proof}

Let us next recall the definition of the {\it tensor square for a non-abelian
group}. For a group $G$, the tensor square $G\otimes G$ is the
group generated by the symbols $g\otimes h,\ g,\,h\in G$, satisfying
 the following defining relations:
\begin{align}
& fg\otimes h=(g^{f^{-1}}\otimes h^{f^{-1}})(f\otimes h)
\label{relation1}\\ & f\otimes gh=(f\otimes g)(f^{g^{-1}}\otimes
h^{g^{-1}}),\label{relation2}
\end{align}
for all $f,\,g,\,h\in G$.
\par\vspace{.25cm}
The {\it exterior square} $$G\wedge G:=G\otimes G/(\langle
g\otimes g\,|\, g\in G\rangle$$ is naturally isomorphic to
$[F,\,F]/[R,\,F]$. Since $G_{ab}\wedge
G_{ab}=\gamma_2(F)/([R,\,F]\gamma_3(F)),$ there is a natural
isomorphism $$
\frac{\gamma_3(F)}{[R,\,F]\cap \gamma_3(F)}\cong\ker\{G\wedge G\to
G_{ab}\wedge G_{ab}\}
$$

\begin{theorem}\label{cubicpicture}{\it (Cubic functors)}
\begin{align*}
& (i)\ \ \ilimit\ \frac{\gamma_3(F)}{[R,\,F,\,F]\gamma_4(F)}=\EuScript
L^3(G_{ab})\\
& (ii)\ \ \ilimit \frac{\gamma_3(F)}{(R\cap
\gamma_3(F))\gamma_4(F)}=\gamma_3(G)/\gamma_4(G)\\
& (iii)\ \ \ilimit \frac{\gamma_3(F)}{([R,\,F]\cap
\gamma_3(F))\gamma_4(F)}=\ker\{G/\gamma_3(G)\wedge
G/\gamma_3(G)\to
G_{ab}\wedge G_{ab}\}\\
& (iv)\ \ \ilimit\
\frac{\gamma_3(F)}{\gamma_3(R)\gamma_4(F)}=L_2\mathfrak
L_s^3(G_{ab})\\
& (v)\ \ \ilimit\ \frac{\gamma_3(F)}{[\gamma_2(R),\,F]\gamma_4(F)}=L_1\operatorname{SP}^3(G_{ab})\\
& (vi)\ \ \ilimit\ \frac{\gamma_3(F)}{[R,\,\gamma_2(F)]\gamma_4(F)}=0
\end{align*}
\end{theorem}
\begin{proof}
The cases $(i)$ - $(iii)$ are obvious, since the representations are
constant.

\par\vspace{.25cm} \noindent $(iv)$ Consider the exact sequence
$$
1\to \frac{D(4,\,\mathfrak r^2\mathfrak f)}{\gamma_3(R)\gamma_4(F)}\to
\frac{\gamma_3(F)}{\gamma_3(R)\gamma_4(F)}\to \frac{\mathfrak
f^3}{\mathfrak r^2\mathfrak f+\mathfrak f^4}
$$
The right hand representation lies in the short exact sequence
$$
0\to \bar R\otimes \bar R\otimes F_{ab}\to F_{ab}^{\otimes
3}\to\frac{\mathfrak f^3}{\mathfrak r^2\mathfrak f+\mathfrak f^4}
\to 0,
$$
where $\bar R=R/R\cap [F,\,F]$. By Proposition  3.7 in  \cite{IvanovMikhailov},
$\ilimit^i \bar R\otimes \bar R\otimes
F_{ab}=\ilimit^iF_{ab}^{\otimes 3}=0,\ i\geq 0$. Hence,
$\ilimit\frac{\mathfrak f^3}{\mathfrak r^2\mathfrak f+\mathfrak
f^4}=0$. Consequently, the statement (iv) follows from Theorem \ref{super}.

\par\vspace{.25cm}\noindent $(v)$ Consider the following
commutative diagram with exact rows and columns
$$\xyma{& & & K \ar@{>->}[d]\\ & & \Lambda^2(\bar R) \otimes \bar F\ar@{=}[r] \ar@{>->}[d] & \Lambda^2(\bar R)\otimes \bar F\ar@{->}[d]\\
\Lambda^3(\bar R) \ar@{>->}[d] & \Lambda^3(\bar
F)\ar@{>->}[r]\ar@{=}[d] & \Lambda^2(\bar
F)\otimes \bar F\ar@{->>}[r]\ar@{->>}[d] & \EuScript L^3(\bar F)\ar@{->>}[d]\\
K \ar@{>->}[r]\ar@{->>}[d] & \Lambda^3(\bar F) \ar@{->}[r] &
\frac{\Lambda^2(\bar F)}{\Lambda^2(\bar R)}\otimes \bar
F\ar@{->>}[r] & \frac{\gamma_3(F)}{[R,\,R,\,F]\gamma_4(F)}\\
L_2\operatorname{SP}^3(G_{ab})}
$$
where $\bar{F}=F_{ab},\  \bar{R}=R/R\cap \gamma_2(F)$ and $K=\ker\{\Lambda^2(\bar R)\otimes F_{ab}\to \EuScript
L^3(F_{ab})\}.$ It follows from Prop. 3.7 in  \cite{IvanovMikhailov},
that
$$
\ilimit \frac{\gamma_3(F)}{[\gamma_2(R),F]\gamma_4(F)}=\ilimit^2\
K=\ilimit^2\ \Lambda^3(\bar R).
$$
The arguments similar to those given in the proof of  Theorem 8.1 in \cite{IvanovMikhailov}
imply that
$$
\ilimit^2\ \Lambda^3(\bar R)=L_1\operatorname{SP}^3(G_{ab})
$$
and thus the asserted  statement  follows.

\par\vspace{.25cm}\noindent  $(vi)$
 Observe that
$[R,\,\gamma_2(F)]\gamma_4(F)=[R\gamma_2(F),\,\gamma_2(F)]\gamma_4(F).$ Setting   $S=R\gamma_2(F)$, there is
an equality
$$
[\gamma_2(F),\,S]\gamma_4(F)=(\gamma_2(S)\cap \gamma_3(F))\gamma_4(F).
$$
We have the following commutative diagram with exact rows and
columns:
$$
\xyma{ & \frac{D(4,\,\mathfrak f\mathfrak
s)}{\gamma_2(S)\gamma_4(F)} \ar@{>->}[d] \ar@{>->}[r]^\theta &
\frac{D(3,\,\mathfrak
f\mathfrak s)}{\gamma_2(S)\gamma_3(F)}\ar@{>->}[d] \\
\frac{\gamma_3(F)}{[S,\,\gamma_2(F)]\gamma_4(F)}\ar@{->}[d] \ar@{>->}[r] &
\frac{\gamma_2(F)}{\gamma_2(S)\gamma_4(F)} \ar@{->}[d]
\ar@{->>}[r] &
\frac{\gamma_2(F)}{\gamma_2(S)\gamma_3(F)}\ar@{->}[d]\\
\frac{\mathfrak f^3}{\mathfrak{fs}\cap\mathfrak{f^3}+\mathfrak
f^4}\ar@{>->}[r] & \frac{\mathfrak f^2}{\mathfrak f\mathfrak
s+\mathfrak f^4}\ar@{->>}[r] & \frac{\mathfrak f^2}{\mathfrak
f\mathfrak s+\mathfrak f^3}}
$$
Since representations $\mathfrak f^2/(\mathfrak f\mathfrak s+f^i),\
i=3,\,4$ are monoadditive,
$$
\ilimit \frac{\mathfrak f^2}{\mathfrak f\mathfrak s+\mathfrak
f^3}=\ilimit \frac{\mathfrak f^2}{\mathfrak f\mathfrak s+\mathfrak
f^4}=0.
$$
Hence there is the following commutative square of monomorphisms
$$
\xyma{\ilimit \frac{D(4,\,\mathfrak f\mathfrak
s)}{\gamma_2(S)\gamma_4(F)} \ar@{=}[d]\ar@{>->}[r]^{\ilimit
\theta}& \ar@{=}[d] \frac{D(3,\,\mathfrak f\mathfrak
s)}{\gamma_2(S)\gamma_3(F)} \ar@{=}[d]\\ \ilimit
\frac{\gamma_2(F)}{\gamma_2(S)\gamma_4(F)}\ar@{->}[r] & \ilimit
\frac{\gamma_2(F)}{\gamma_2(S)\gamma_3(F)}.}
$$
The middle horizontal exact sequence implies the following exact
sequence of limits:
$$
1\to \ilimit \frac{\gamma_3(F)}{[R,\,\gamma_2(F)]\gamma_4(F)}\to \ilimit
\frac{\gamma_2(F)}{\gamma_2(S)\gamma_4(F)}\to \ilimit
\frac{\gamma_2(F)}{\gamma_2(S)\gamma_3(F)}
$$
Since the right hand map is a monomorphism, the stated claim follows. \end{proof}\par\vspace{.25cm}
\section*{Acknowledgement}
The research of the first author is supported by the Russian
Science Foundation, grant N 14-21-00035. The results reported in this  paper
were obtained during his  visit  to Indian Institute of Science Education and Reseach
Mohali; he   wishes to express his gratitude to the
Institute for its  warm hospitality.

\par\vspace{.5cm}\noindent
Roman Mikhailov\\
St Petersburg Department of Steklov Mathematical Institute\\
and\\
Chebyshev Laboratory\\
St Petersburg State University\\
14th Line, 29b\\
Saint Petersburg\\
199178 Russia\\
email: romanvm@mi.ras.ru
\par\vspace{.5cm}\noindent
Inder Bir S. Passi\\
Centre for Advanced Study in Mathematics\\
Panjab University\\
Sector 14\\
Chandigarh 160014 India\\
and\\
Indian Institute of Science Education and Research\\
Mohali (Punjab)140306 India\\
email: ibspassi@yahoo.co.in
\end{document}